\documentclass[11pt,reqno]{amsart}
\usepackage{latexsym}
\RequirePackage{amsmath}   \RequirePackage{amssymb}
\RequirePackage{amsthm}   \RequirePackage{epsfig}
\theoremstyle{plain}
\newtheorem{theorem}{Theorem}

\newtheorem*{corollary*}{Corollary}
\newtheorem{theoremO}{Theorem}
\newtheorem{lemmO}[theoremO]{Lemma}

\newtheorem*{conjecture*}{Conjecture}

\theoremstyle{definition}

\newtheorem*{problem}{Problem}

\theoremstyle{remark}

\newcommand{\SC}{{\mathbb C}}  \newcommand{\SD}{{\mathbb D}}  
    
  \newcommand{\ga}{\gamma}  
    
    \newcommand{\la}{\lambda}
    \newcommand{\om}{\omega}
\newcommand{\De}{\Delta}

\newcommand{\be}{\begin{equation}}
\newcommand{\ee}{\end{equation}}
\newcommand{\bea}{\begin{eqnarray}}
\newcommand{\eea}{\end{eqnarray}}

\begin{document}

\title[Quasiconformal extension for harmonic mappings]{Quasiconformal extension for harmonic mappings on finitely connected domains}  

\author[I. Efraimidis]{Iason Efraimidis}
\address{Department of Mathematics and Statistics, Texas Tech University, Box 41042, Lubbock, TX 79409, United States} \email{iason.efraimidis@ttu.edu}  

\subjclass[2010]{30C55, 30C62, 31A05} 

\keywords{Harmonic mapping, Schwarzian derivative, quasiconformal extension, quasiconformal decomposition}

\maketitle

\begin{abstract}
We prove that a harmonic quasiconformal mapping defined on a finitely connected domain in the plane, all of whose boundary components are either points or quasicircles, admits a quasiconformal extension to the whole plane if its Schwarzian derivative is small. We also make the observation that a univalence criterion for harmonic mappings holds on uniform domains. 
\end{abstract}

\section{ Introduction} 

Let $f$ be a harmonic mapping in a planar domain $D$ and let $\om=\overline{f_{\overline{z}}} / f_z$ be its dilatation. According to Lewy's theorem the mapping $f$ is locally univalent if and only if its Jacobian $J_f=|f_z|^2-|f_{\overline{z}}|^2$ does not vanish. Duren's book \cite{Du} contains valuable information about the theory of planar harmonic mappings. 

The Schwarzian derivative of $f$ was defined by Hern\'andez and Mart\'in \cite{HM15} as 
\begin{equation} \label{Schw-def}
S_f \, = \, \rho_{zz} - \tfrac{1}{2} (\rho_z)^2, \qquad \text{where} \qquad \rho =\log J_f. 
\end{equation}
When $f$ is holomorphic this reduces to the classical Schwarzian derivative. Another definition, introduced by Chuaqui, Duren and Osgood \cite{CDO03}, applies to harmonic mappings which admit a lift to a minimal surface via the Weierstrass-Enneper formulas. However, focusing on the planar theory in this note we adopt the definition \eqref{Schw-def}.

We assume that $\overline{\SC}\backslash D$ contains at least three points, so that $D$ is equipped with the hyperbolic metric, defined by 
$$
\la_D \big(\pi(z)\big) |\pi'(z)| |dz| \, = \, \la_\SD(z) |dz|\, = \,\frac{|dz|}{1-|z|^2}, \qquad z\in \SD, 
$$
where $\SD$ is the unit disk and $\pi:\SD\to D$ is a universal covering map. The size of the Schwarzian derivative of a mapping $f$ in $D$ is measured by the norm
$$
\|S_f\|_D \, = \, \sup_{z\in D} \, \la_D(z)^{-2} |S_f(z)|. 
$$

A domain $D$ in $\overline{\SC}$ is a $K$-quasidisk if it is the image of the unit disk under a $K$-quasiconformal self-map of $\overline{\SC}$, for some $K\geq 1$. The boundary of a quasidisk is called a quasicircle. 

According to a theorem of Ahlfors~\cite{Ah}, if $D$ is a $K$-quasidisk then there exists a constant $c>0$, depending only on $K$, such that if $f$ is analytic in $D$ with $\|S_f\|_D \leq c$ then $f$ is univalent in $D$ and has a quasiconformal extension to $\overline{\SC}$. This has been generalized by Osgood~\cite{O80} to the case when $D$ is a finitely connected domain whose boundary components are either points or quasicircles. Further, the univalence criterion was generalized to uniform domains (see Section~\ref{sect-uniform} for a definition) by Gehring and Osgood~\cite{GO79} and, subsequently, the quasiconformal extension criterion was generalized to uniform domains  by Astala and Heinonen~\cite{AH88}. 

For harmonic mappings and the definition \eqref{Schw-def} of the Schwarzian derivative, a univalence and quasiconformal extension criterion in the unit disk $\SD$ was proved by Hern\'andez and Mart\'in~\cite{HM15-2}. This was recently generalized to quasidisks by the present author in \cite{Ef}. Moreover, in \cite{Ef} it was shown that if all boundary components of a finitely connected domain $D$ are either points or quasicircles then any harmonic mapping in $D$ with sufficiently small Schwarzian derivative is injective. The main purpose of this note is to prove the following theorem. 

\begin{theorem}\label{main-thm}
Let $D$ be a finitely connected domain whose boundary components are either points or quasicircles and let also $d\in[0,1)$. Then there exists a constant $c>0$, depending only on the domain $D$ and the constant $d$, such that if $f$ is harmonic in $D$ with $\|S_f\|_D \leq c$ and with dilatation $\om$ satisfying $|\om(z)| \leq d$ for all $z\in D$ then $f$ admits a quasiconformal extension to $\overline{\SC}$. 
\end{theorem}

As mentioned above, for the case when $D$ is a (simply connected) quasidisk this was shown in \cite{Ef} while, on the other hand, for the case $d=0$ (when $f$ is analytic) this was proved by Osgood in \cite{O80}. Osgood's proof amounts to proving a univalence criterion in $f(D)$. Such an approach does not seem to work here since for a holomorphic $\phi$ on $f(D)$ the composition $\phi\circ f$ is not, in general, harmonic. 

Since isolated boundary points are removable for quasiconformal mappings (see \cite[Ch.I, \S8.1]{LV}), we may assume for the proof of Theorem~\ref{main-thm} that $\partial D$ consists of $n$ non-degenerate quasicircles. Our proof will be based on the following theorem of Springer~\cite{Sp64} (see also \cite[Ch.II, \S8.3]{LV}).

\begin{theoremO}[\cite{Sp64}]\label{Springer-thm}
Let $D$ and $D'$ be two $n$-tuply connected domains whose boundary curves are quasicircles. Then every quasiconformal mapping of $D$ onto $D'$ can be extended to a quasiconformal mapping of the whole plane.
\end{theoremO}

Hence, to prove Theorem~\ref{main-thm} it suffices that we show that the boundary components of $f(D)$ are quasicircles. We prove this in Section~\ref{sect-proof}. It relies on Osgood's \cite{O80} quasiconformal decomposition, which we briefly present in Section~\ref{sect-QC-decomp}. In Section~\ref{sect-uniform} we give a univalence criterion on uniform domains.

\section{Quasiconformal Decomposition} \label{sect-QC-decomp}

Let $D$ be a domain in $\overline{\SC}$. A collection $\mathfrak{D}$ of domains $\De \subset D$ is called a $K$-quasiconformal decomposition of $D$ if each $\De$ is a $K$-quasidisk and any two points $z_1,z_2\in D$ lie in the closure of some $\De\in\mathfrak{D}$. This definition was introduced by Osgood in \cite{O80}, along with the following lemma. 

\begin{lemmO}[\cite{O80}] \label{lem-qc-decomp}
If $D$ is a finitely connected domain and each component of $\partial D$ is either a point or a quasicircle then $D$ is quasiconformally decomposable. 
\end{lemmO}

We now present, almost verbatim, the construction proving Lemma~\ref{lem-qc-decomp}. We focus on the parts of the construction we will be needing, maintaining the notation of \cite{O80} and skipping all the relevant proofs. The interested reader should consult \cite{O80} for further details. 

As we mentioned earlier, we may assume that $\partial D$ consists of non-degenerate quasicircles $C_0, C_1, \ldots, C_{n-1}$, for $n\geq2$. Let $F$ be a conformal mapping of $D$ onto a circle domain $D'$. Then, with an application of Theorem~\ref{Springer-thm} to $F^{-1}$, it will be sufficient to find a quasiconformal decomposition of $D'$. Hence we may assume that $D$ itself is a circle domain with boundary circles $C_j, \, j=0,\ldots,n-1$. 

If $n=2$ then we may assume that $D$ is the annulus $1<|z|<R$. Then the domains 
$$
\De_1=\{z\in D \,:\, 0<{\rm arg}(z)<\tfrac{4\pi}{3}\}, \qquad \De_2=e^{2\pi i /3} \De_1, \qquad \De_3=e^{4\pi i /3} \De_1
$$
make a quasiconformal decomposition of $D$. 

Let $n\geq3$. Then there exists a conformal mapping $\Psi$ of the circle domain $D$ onto a domain $D'$ consisting of the entire plane minus $n$ finite rectilinear slits lying on rays emanating from the origin. The mapping can be chosen so that no two distinct slits lie on the same ray. The boundary behavior of $\Psi$ is the following: it can be analytically extended to $\overline{D}$, and the two endpoints of the slit $C_j'=\Psi(C_j)$ correspond to two points on the circle $C_j$ which partition $C_j$ into two arcs, each of which is mapped onto $C_j'$ in a one-to-one fashion. 

Let $\xi_j$ be the endpoint of $C_j'$ furthest from the origin and let $Q_j'$ be the part of the ray that joins $\xi_j$ to infinity. Let also $S_j'$ be the sector between $C_j'$ and $C_{j+1}'$. Let $\om_j'$ be the midpoint of $C_j'$ and let $P_j'$ be a polygonal arc joining $\om_j'$ to $\om_{j+1}'$ that, except for its endpoints, lies completely in $S_j'$. Then
$$
P'=\bigcup_{j=0}^{n-1} P_j'
$$
is a closed polygon separating $0$ from $\infty$ that does not intersect any of the $Q_j'$. Let $G_0'$ and $G_1'$ be the components of $D'\backslash P'$ that contain $0$ and $\infty$, respectively. Now define
$$
\Delta_{0j}' = G_0' \cup S_j', \qquad \Delta_j' = D' \backslash \Delta_{0j}' 
$$
and
$$
\mathfrak{D}' = \{\Delta_{0j}' \, : \, j=0,1,\ldots,n-1 \} \cup \{\Delta_j' \, : \, j=0,1,\ldots,n-1 \}. 
$$
This collection has the covering property for $D'$. 

We denote the various parts of $D$ corresponding under $\Psi^{-1}$ to those of $D'$ by the same symbol without the prime. Then 
$$
\mathfrak{D} = \{\Delta_{0j} \, : \, j=0,1,\ldots,n-1 \} \cup \{\Delta_j \, : \, j=0,1,\ldots,n-1 \} 
$$
is a quasiconformal decomposition of $D$.

\section{Proof of Theorem~\ref{main-thm}} \label{sect-proof}

Let $f$ be a mapping in $D$ as in Theorem~\ref{main-thm}. By Theorem~2 in \cite{Ef}, $f$ is injective if $c$ is sufficiently small. Also, $f$ extents continuously to $\partial D$ since every boundary point of $D$ belongs to $\partial \Delta$ for some $\Delta$ in the collection $\mathfrak{D}$ and, by Theorem~1 in \cite{Ef}, the restriction of $f$ on $\Delta$ admits a homeomorphic extension to $\overline{\SC}$. 

Let $\Psi$ be a conformal mapping of $D$ onto the slit domain $D'$ of the previous section. Let $C_j$ be a boundary quasicircle of $D$. We first prove that $f(C_j)$ is a Jordan curve. The slit $C_j'$ is divided by its midpoint $\om_j'$ into two line segments, which we denote by $\Sigma_j'(m), m=1,2$, so that 
$$
\Sigma_j'(1)=\{z\in C_j' : |z|\leq|\om_j'|\} \qquad \text{and} \qquad \Sigma_j'(2)=\{z\in C_j' : |z|\geq|\om_j'|\}. 
$$
Let $\Sigma_j'(m)^{\pm}$ denote the two sides of $\Sigma_j'(m)$, so that a point $z_0$ on $\Sigma_j'(m)^-$ is reached only by points $z \in S_{j-1}'$, meaning that $\arg z \to (\arg z_0)^-$ when $z\to z_0$. Similarly, a point $z_0$ on $\Sigma_j'(m)^+$ is reached only by points $z \in S_j'$, so that $\arg z \to (\arg z_0)^+$ when $z\to z_0$. Corresponding under $\Psi^{-1}$ are four disjoint -except for their endpoints- arcs on the quasicircle $C_j$, denoted without the prime by $\Sigma_j(m)^{\pm}, m=1,2$. Now consider the domains $\Delta_{0,j-1}, \Delta_{0j}$ and $\Delta_k$ in the collection $\mathfrak{D}$, for some $k\neq j-1, j$; see Figure~\ref{fig} for their images under $\Psi$. By Theorem~1 in \cite{Ef} $f$ is injective up to the boundary of each $\Delta\in\mathfrak{D}$. Note that the arcs $\Sigma_j(1)^-, \Sigma_j(1)^+$ and $\Sigma_j(2)^-$ are subsets of $\partial\Delta_{0,j-1}$, so that their images under $f$, except for their endpoints, are disjoint. It remains to show that the images of these three arcs under $f$ are not intersected by the remaining image $f(\Sigma_j(2)^+)$. Note that the arcs $\Sigma_j(1)^-, \Sigma_j(1)^+$ and $\Sigma_j(2)^+$ are subsets of $\partial\Delta_{0j}$, so that $f(\Sigma_j(2)^+)$ does not intersect $f(\Sigma_j(1)^-)$ nor $f(\Sigma_j(1)^+)$. What remains to be seen is that $f(\Sigma_j(2)^-)$ and $f(\Sigma_j(2)^+)$ are disjoint and this follows from the fact that the arcs $\Sigma_j(2)^-$ and $\Sigma_j(2)^+$ are subsets of $\partial\Delta_k$.

\begin{figure}  
\includegraphics[scale=.44]{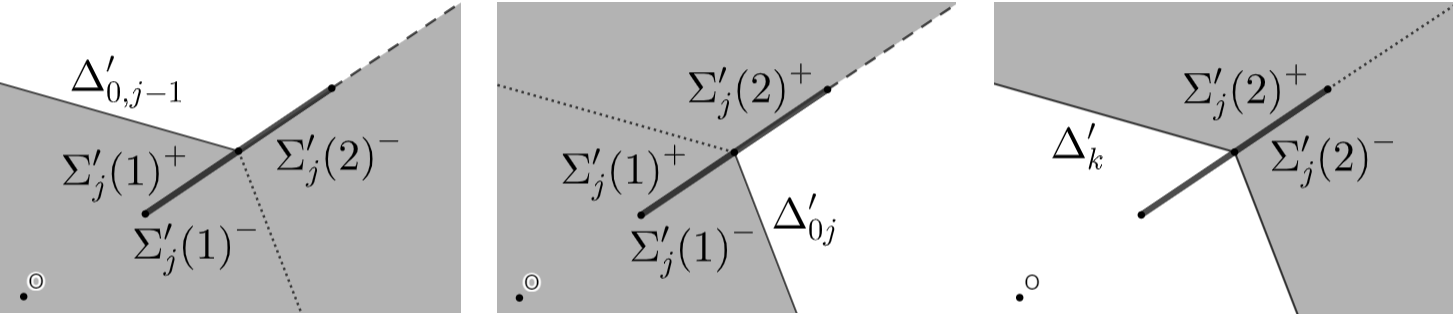} 
\caption{The slit $C_j'$ and the three distinguished domains.} \label{fig}
\end{figure}

To see that the Jordan curve $f(C_j)$ is actually a quasicircle note that each point of $f(C_j)$ belongs to some open subarc of $f(C_j)$ which is entirely included in the boundary of either $f(\Delta_{0,j-1}), f(\Delta_{0j})$ or $f(\Delta_k)$. These three domains are quasidisks by Theorem~3 in \cite{Ef}. Now the assertion that $f(C_j)$ is a quasicircle follows by an application of Theorem~8.7 in \cite[Ch.II, \S8.9]{LV}.

\section{Remarks on uniform domains} \label{sect-uniform}  

A domain $D$ in $\SC$ is called uniform if there exist positive constants $a$ and $b$ such that each pair of points $z_1,z_2\in D$ can be joined by an arc $\ga \subset D$ so that for each $z\in \ga$ it holds
$$
\ell(\ga) \leq a \, |z_1-z_2|
$$
and 
$$
\min_{j=1,2}\ell(\ga_j) \leq b \, {\rm dist}(z,\partial D),
$$
where $\ga_1, \ga_2$ are the components of $\ga\backslash\{z\}$, ${\rm dist}(z,\partial D)$ denotes the euclidean distance from $z$ to the boundary of $D$ and $\ell(\cdot)$ denotes euclidean length. Uniform domains were introduced by Martio and Sarvas \cite{MS79}; see also, \emph{e.g.}, \cite{GO79} for this equivalent definition. In \cite{MS79} it was shown that all boundary components of a uniform domain are either points or quasicircles. The converse of this is also true for finitely connected domains, but not, in general, for domains of infinite connectivity; see \cite[\S3.5]{GeHa}. The following univalence criterion was proved in \cite{MS79}.

\begin{theoremO}[\cite{MS79},\cite{GO79}] \label{Thm-uniform-1-1}
If $D$ is a uniform domain then there exists a constant $c>0$ such that every analytic function $f$ in $D$ with $\|S_f\|_D\leq c$ is injective. 
\end{theoremO}

Gehring and Osgood \cite{GO79} gave a different proof of Theorem~\ref{Thm-uniform-1-1} by providing a characterization of uniform domains. They showed that a domain $D$ is uniform if and only if it is quasiconformally decomposable in the following weaker (than the one we saw in Section~\ref{sect-QC-decomp}) sense: there exists a constant $K$ with the property that for each $z_1,z_2\in D$ there exists a $K$-quasidisk $\De\subset D$ for which $z_1,z_2\in\overline{\De}$. Note that, in contrast to Osgood's \cite{O80} decomposition, here $\De$ depends on the points $z_1,z_2$. However, this can readily be used to generalize the implication (i) $\Rightarrow$ (iii) of Theorem~2 in \cite{Ef}, according to which a univalence criterion for harmonic mappings holds on finitely connected uniform domains. The following theorem extends it to all uniform domains. 

\begin{theorem}
Let $D$ be a uniform domain in $\SC$. Then there exists a constant $c>0$ such that if $f$ is harmonic in $D$ with $\|S_f\|_D \leq c$ then $f$ is injective. 
\end{theorem} 
\begin{proof}
Assume that there exist distinct points $z_1,z_2\in D$ for which $f(z_1)=f(z_2)$. By \cite{GO79}, there exists a $K$-quasidisk $\De\subset D$ for which $z_1,z_2\in\overline{\De}$. The domain monotonicity for the hyperbolic metric shows that 
$$
\|S_f\|_\De \leq \|S_f\|_D \leq c. 
$$
But the homeomorphic extension of Theorem~1 in \cite{Ef} shows that if $c$ is sufficiently small then $f$ is injective up to the boundary of $\De$, a contradiction. 
\end{proof}

Regarding quasiconformal extension, Astala and Heinonen \cite{AH88} proved the following theorem. 

\begin{theoremO}[\cite{AH88}] \label{Thm-uniform-qc-ext}
If $D$ is a uniform domain then there exists a constant $c>0$ such that every analytic function $f$ in $D$ with $\|S_f\|_D\leq c$ admits a quasiconformal extension to $\overline{\SC}$. 
\end{theoremO}

This evidently implies Theorem~\ref{Thm-uniform-1-1} and was also proved in substantially greater generality, but we omit it here. It is not clear how to generalize Theorem~\ref{Thm-uniform-qc-ext} to the setting of harmonic mappings. Therefore, we propose the following problem. 

\begin{problem}
Let $D$ be a uniform domain. Does there exist a constant $c>0$ such that if $f$ is harmonic in $D$ with $\|S_f\|_D \leq c$ and with dilatation $\om$ satisfying $\sup_{z\in D}|\om(z)| <1$ then $f$ admits a quasiconformal extension to $\overline{\SC}$? 
\end{problem}


\begin{thebibliography}{99}
\bibitem{Ah} L. Ahlfors, Quasiconformal reflections, \emph{Acta Math.} \textbf{109} (1963), 291-301.  

\bibitem{AH88} K. Astala, J. Heinonen, On quasiconformal rigidity in space and plane, \emph{Ann. Acad. Sci. Fenn. Ser. A I Math.} \textbf{13} (1988), no. 1, 81-92.

\bibitem{CDO03} M. Chuaqui, P. Duren, B. Osgood, The Schwarzian derivative for harmonic mappings, \textit{J. Anal. Math.} \textbf{91} (2003), 329-351. 

\bibitem{Du} P.L. Duren, \emph{Harmonic Mappings in the Plane}, Cambridge University Press, Cambridge, 2004.

\bibitem{Ef} I. Efraimidis, Criteria for univalence and quasiconformal extension for harmonic mappings on planar domains, preprint, arXiv:2009.14766. 

\bibitem{GeHa} F.W. Gehring, K. Hag, \emph{The ubiquitous quasidisk}, Amer. Math. Soc., Providence, RI, 2012. 

\bibitem{GO79} F.W. Gehring, B.G. Osgood, Uniform domains and the quasihyperbolic metric, \emph{J. Analyse Math.} \textbf{36} (1979), 50-74 (1980).

\bibitem{HM15} R. Hern\'andez, M.J. Mart\'in, Pre-Schwarzian and Schwarzian derivatives of harmonic mappings, \textit{J. Geom. Anal.} \textbf{25} (2015), no.~1, 64-91. 

\bibitem{HM15-2} R. Hern\'andez, M.J. Mart\'in, Criteria for univalence and quasiconformal extension of harmonic mappings in terms of the Schwarzian derivative, \textit{Arch. Math. (Basel)} \textbf{104} (2015), no.~1, 53-59. 

\bibitem{LV} O. Lehto, K.I. Virtanen, \emph{Quasiconformal mappings in the plane}, second edition, Springer-Verlag, New York-Heidelberg, 1973.  

\bibitem{MS79} O. Martio, J. Sarvas, Injectivity theorems in plane and space, \emph{Ann. Acad. Sci. Fenn. Ser. A I Math.} \textbf{4} (1979), no. 2, 383-401. 

\bibitem{O80} B. Osgood, Univalence criteria in multiply-connected domains, \emph{Trans. Amer. Math. Soc.} \textbf{260} (1980), no.~2, 459-473. 

\bibitem{Sp64} G. Springer, Fredholm eigenvalues and quasiconformal mapping, \emph{Acta Math.} \textbf{111} (1964), 121-142.
\end{thebibliography}
\end{document}